\documentclass[preprint,12pt]{elsarticle}
\usepackage{amsthm,amsfonts,amssymb,amscd,amsmath,enumerate,verbatim,calc,graphicx,geometry}
\usepackage[all]{xy}
\newtheorem{theorem}{Theorem}[section]

\newtheorem{proposition}[theorem]{Proposition}
\newtheorem{corollary}[theorem]{Corollary}
\theoremstyle{definition}
\theoremstyle{definitions}
\newtheorem{definition}[theorem]{Definition}

\newtheorem{remark}[theorem]{Remark}
\newtheorem{example}[theorem]{Example}
\theoremstyle{notations}

\theoremstyle{remarks}

\newcommand{\N}{\mathbb{N}}

\newcommand{\sub}{\subseteq}

\newcommand{\ov}{\overline}

\newcommand{\lo}{\longrightarrow}

\newcommand{\wt}{\widetilde}

\newcommand{\al}{\alpha}

\newcommand{\bt}{\beta}
\newcommand{\ti}{\tilde}
\newcommand{\tx}{\textit}

\newcommand{\lk}{\langle}
\newcommand{\rg}{\rangle}
\newcommand{\psg}{\pi_1^{sg}(X,x)}
\newcommand{\pt}{\pi_1^{qtop}(X,x)}
\newcommand{\px}{p:(\wt{X}, \ti{x})\longrightarrow (X,x)}
\newcommand{\ps}{\pi_1^s(X,x)}
\newcommand{\pst}{p_*\pi_1(\wt{X},\ti{x})}

\journal{Mathematica Slovaca}

\begin{document}

\begin{frontmatter}



\title{Topological Fundamental Groups and Small Generated Coverings}


\author[]{Hamid~Torabi}
\ead{hamid$_{-}$torabi86@yahoo.com}
\author[]{Ali~Pakdaman}
\ead{Alipaky@yahoo.com}
\author[]{Behrooz~Mashayekhy\corref{cor1}}
\ead{bmashf@um.ac.ir}
\address{Department of Pure Mathematics, Center of Excellence in Analysis on Algebraic Structures, Ferdowsi University of Mashhad,\\
P.O.Box 1159-91775, Mashhad, Iran.}
\cortext[cor1]{Corresponding author}
\begin{abstract}
This paper is devoted to study some topological properties of the SG subgroup, $\pi_1^{sg}(X,x)$, of the quasitopological fundamental group of a based space $(X,x)$, $\pt$, its topological properties as a subgroup
of the topological fundamental group $\pi_1^{\tau}(X,x)$ and its influence on the existence of universal covering of $X$. First, we introduce small generated spaces which have indiscrete topological fundamental groups and also small generated coverings which are universal coverings
 in the categorical sense. Second, we give a necessary and sufficient condition for the existence of the small generated coverings.
 Finally, by introducing the notion of semi-locally small generatedness we show that the quasitopological fundamental groups of semi-locally small generated spaces are topological groups.

\end{abstract}


\end{frontmatter}


\section{Introduction and Motivation}
As it is shown in \cite{P2}, there exist special subgroups of fundamental groups of non-homotopically Hausdorff spaces which have great influence on their coverings. In fact, if a space $X$ is not homotopically Hausdorff, then there exists $x\in X$ and a nontrivial loop in $X$ based at $x$ which is homotopic to a loop in every neighborhood $U$ of $x$ (see \cite{FZ} for the definition of homotopically Hausdorffness). Z. Virk \cite{V} called these loops as small loops and showed that for every $x\in X$ they form a subgroup of $\pi_1(X,x)$ which is named small loop group and denoted by $\pi_1^s(X,x)$. In general, various points of $X$ have different small loop groups and hence in order to have a subgroup independent of the base point, Virk \cite{V} introduced the SG (small generated) subgroup, denoted by $\psg$, as the subgroup generated by the following set
$$\{[\al*\bt*\al^{-1}]\ |\ [\bt]\in\pi_1^s(X,\al(1)),\ \al\in P(X,x)\},$$
where $P(X,x)$ is the space of all paths from $I$ into $X$ with initial point $x$.
Virk \cite{V} calls a space $X$ a small loop space if $\pi_1^s(X,x)= \pi_1(X,x)\neq 1$ for all $x\in X$.
The authors \cite{P2} showed that for a covering $p:(\wt{X}, \ti{x})\rightarrow (X,x)$ the following relations hold:
$$\pi_1^s(X,x)\leq\psg\leq p_*\pi_1(\wt{X},\ti{x}).$$
It should be noted that by a result of Spanier \cite[\S 2.5 Lemma 11]{S} one has
$$\psg\leq \pi(\mathcal{U},x)\leq p_*\pi_1(\wt{X},\ti{x}),\ \ \ (*)$$ where $\mathcal{U}$ is any open cover of $X$ by evenly covered open sets and $\pi(\mathcal{U},x)$ is the subgroup of $\pi_1(X,x)$ generated by all elements of the form $[\al*\bt*\al^{-1}]$, for all $\al\in P(X,x)$ and $[\bt]\in\pi_1(U,\al(1))$ for some $U\in \mathcal{U}$.
We also showed \cite{P2} that if $\wt{X}$ is a small loop space, then a covering $\px$ is the universal covering for which $\pst=\ps$ and called it {\em small covering}. Moreover, the authors \cite{P2} showed that a connected and locally path connected space has a small covering if and only if it is a semi-locally small loop space. A space is called a semi-locally small loop space if for every $x\in X$ there exists an open neighborhood $U$ of $x$ such that $i_*\pi_1(U,y)=\pi_1^s(X,y)$, for all $y\in U$, where $i:U\hookrightarrow X$ is the inclusion map. If $X$ is a semi-locally small loop space, then for every $x\in X$ we have $\ps=\psg$ and so small loop groups are isomorphic for different base points \cite[Lemma 4.2]{P2}.
Since the SG subgroup does not depend on the base point even for spaces with various small loop groups, we are interested in finding out some conditions which guaranty the existence of a covering $\px$ with $\pst=\psg$.

Endowing a topology on the familiar fundamental group $\pi_1(X,x)$ as a quotient of the loop space $\Omega(X,x)$ equipped with the compact-open topology
with respect to the canonical map $\Omega(X,x)\lo \pi_1(X,x)$ identifying path components, the quasitopological fundamental group $\pt$ was introduced by D. Biss \cite{B}. It should be mentioned that $\pt$ is a quasitopological group in the sense of \cite{A}, that is, a group with a topology such that inversion and all translations are continuous, and it is not always a topological group (see \cite{Br,F}). Also, the topological fundamental group $\pi_1^{\tau}(X,x)$ is the fundamental group endowed with another topology introduced by J. Brazas \cite {Br2}. In fact, Brazas gives a topology to $\pi_1^{\tau}(X,x)$ by removing the smallest number of open sets from the topology of $\pt$ in order to make it a topological group.

The existence of a covering $\px$ with $\pst=\psg$ is depended on the topology of $\psg$ in $\pt$ (see correction of \cite[Theorem 5.5]{B} in \cite{T}):\\
(1.1) \emph{ Given a connected and locally path connected space $X$ and a subgroup $H$ of $\pi_1(X,x)$, there is a covering $p:\wt{X}\lo X$ with $p_*\pi_1(\wt{X},\ti{x})=H$  if and only if $H$ contains a normal subgroup of $\pi_1(X,x)$ which is open in $\pi_1^{qtop}(X,x)$; in which case $H$ itself is open in  $\pi_1^{qtop}(X,x)$.}\\
Since $\pi_1^{sg}(X,x)$ is normal in $\pi_1(X,x)$, there is a covering $p:\wt{X}\lo X$ with $p_*\pi_1(\wt{X},\ti{x})=\pi_1^{sg}(X,x)$ if and only if $\pi_1^{sg}(X,x)$ is an open subgroup of $\pi_1^{qtop}(X,x)$. Moreover, in this case $\pi_1(\wt{X},\ti{x})=\pi_1^{sg}(\wt{X},\wt{x})$. (cf. \cite[Theorem 4.7]{P2}).

In this paper, we call such coverings as {\em small generated covering}. At first, by showing that every open neighborhood of the trivial element in $\pt$ contains $\psg$, we conclude that topological fundamental groups of {\em small generated spaces} are indiscrete topological groups. A space $X$ is said to be small generated if $\pi_1(X,x)=\psg $ for all $x\in X$. Also, we show that every nonempty open subset of $\pt$ is a union of some cosets of the normal subgroup $\psg$ and hence the (quasi)topological fundamental groups of non-homotopically Hausdorff spaces will be described somehow. Furthermore, by some examples, we show that $\psg$ is not necessarily an open or a closed subgroup.

In Section 3, we introduce small generated coverings and show that they are universal coverings in the categorical sense, that is, a covering
$p:\wt{X}\lo X$ with the property that for every covering $q:\wt{Y}\lo X$ with a path connected space $\wt{Y}$ there exists a unique covering $f:\wt{X}\lo\wt{Y}$ such that $q\circ f= p$. Moreover, we find the necessary and sufficient condition \emph{semi-locally small generatedness} for the existence of small generated coverings which is an answer to the question at the end of \cite{P2}. We call a space $X$ semi-locally small generated if for every $x\in X$ there exists an open neighborhood $U$ of $x$ such that $i_*\pi_1(U,x)\leq \psg$. In fact, we show that in a connected and locally path connected space $X$, semi-locally small generatedness is equivalent to the property that a subset $U\sub\pt$ is open if and only if $U$ is a union of some cosets of $\psg$, for every $x\in X$.

It has been shown that the group multiplication in the quasitopological fundamental group introduced by Biss \cite{B} is not necessarily continuous (see \cite{Br,F}). Therefore, it seems interesting to find out when quasitopological fundamental groups are topological groups. In Section 4, we prove that quasitopological fundamental groups of semi-locally small generated spaces are topological groups.

Throughout this article, all the homotopies between two paths are relative to end points,
$X$ is a topological space with the base point $x\in X$, and $p:\wt{X}\lo X$ is a covering of $X$ with $\ti{x}\in p^{-1}(\{x\})$ as the base point of $\wt{X}$.

\section{Topology of the Small Generated Subgroup}
The SG subgroup of the fundamental group of a space $X$ first was introduced by Virk \cite{V} and is the subgroup of $\pi_1(X,x)$ generated by the set $\{[\al*\bt*\al^{-1}]\ |\ [\bt]\in\pi_1^s(X,\al(1)),\ \al\in P(X,x)\}$. It is shown that $\psg$ is a normal subgroup of $\pi_1(X,x)$ and it is point free, that means, for every $x,y\in X$, $\psg\cong\pi_1^{sg}(X,y)$ \cite{V}. Also, since the presence of small loops is equivalent to the absence of homotopically Hausdorffness, a space $X$ is homotopically Hausdorff if and only if $\psg= 1$.

\begin{definition}
For a topological space $X$, a loop $\al:I\lo X$ based at $x$ is called small generated if $[\al]\in\psg$.
\end{definition}

Since the homotopy classes of small loops have a representative in every neighborhood of their base point, they belong to the topological closure of the homotopy class of the constant loop in the quasitopological fundamental group. For, if $\mathcal{W}=\bigcap_{i=1}^n\lk K_i,U_i\rg$ is a basis open neighborhood of the constant path $e_x$ in the space of $x$ based loops in $X$, $\Omega(X,x)$, then $U=\bigcap_{i=1}^nU_i$ is an open neighborhood of $x$. For the small loop $\al$ based at $x$ there exists a loop $\al_U:I\lo U$ homotopic to $\al$ which implies that every open neighborhood of the trivial element in $\pt$ contains $\ps$. Using this fact, the authors \cite{P1} proved that small loop spaces have indiscrete quasitopological fundamental group. Biss \cite{B} showed that the Harmonic Archipelago has also indiscrete quasitopological fundamental group. But, for Harmonic Archipelago we have $\pi_1(HA,0)=\pi_1^{sg}(HA,0)$ and hence it seems that the homotopy class of small generated loops do also belong to the closure of the trivial element in the topological fundamental group, as it is shown in the following theorem.
\begin{theorem}
If $(X,x)$ is a pointed topological space and $U$ is an open neighborhood of the identity element $[e_{x}]\in\pi_1^{qtop}(X,x)$, then $\pi_1^{sg}(X,x)\sub U$. Moreover $\pi_1^{sg}(X,x)\sub \ov{\{[e_{x}]\}}$.
\end{theorem}
\begin{proof}
First we show that every open neighborhood of $[e_x]$ contains every generator of $\pi_1^{sg}(X,x)$. For this, let $W$ be an open neighborhood of $[e_x]$ and $[\al*\bt*\al^{-1}]$ be a generator of $\pi_1^{sg}(X,x)$. Since $[\al*\al^{-1}]=[e_x]$, $\al*\al^{-1}\in q^{-1}(W)$, where $q:\Omega(X,x)\lo\pi_1(X,x)$ is the quotient map with $q(\al)=[\al]$. Hence there exists a basic open neighborhood $\mathcal{U}=\bigcap_{i=1}^n\lk K_i,U_i\rg$ of $\al*\al^{-1}$ in $\Omega(X,x)$ such that $\mathcal{U}\sub q^{-1}(W)$.

 Let $A=\{i\in\{1,2,...,n\}|\ 1/2\in K_i\}$ and $B=\{i\in\{1,2,...,n\}|\ 1/2\notin K_i\}$. Since $\bigcup_{i\in B}K_i$ is compact, there exists $\delta_1>0$ such that for every $i\in B$, $[1/2-\delta_1,1/2+\delta_1]\cap K_i=\varnothing$.
If $A\neq\varnothing$, then $V=\bigcap_{i\in A}U_i$ is a nonempty open subset of $X$ that contains $\al(1)=\bt(0)=(\al*\al^{-1})(1/2)$. Choose $\delta_2>0$ such that $[1-2\delta_2,1]\sub \al^{-1}(V)$. Since $[\bt]\in\pi_1^s(X,\al(1))$, there exists a loop $\bt':I\lo V$ such that $[\bt]=[\bt']$. If $A=\varnothing$, put $\delta_2=1/2$ and $\bt'=\bt$. Define $f:I\lo X$ by
\begin{displaymath}
f(t)= \left\{
\begin{array}{lr}
\al(2t)                     &       0\leq t\leq 1/2 \\
\beta'\circ\varphi_1(t)     &      1/2\leq t\leq 1/2+\delta/2\\
\al^{-1}(2\varphi_2(t)-1)   &  1/2+\delta/2\leq t\leq 1/2+\delta\\
\al^{-1}(2t-1)              &  1/2+\delta\leq t\leq 1,
\end{array}
\right.
\end{displaymath}
where $\delta=min\{\delta_1,\delta_2\}$, $\varphi_1:[1/2,1/2+\delta/2]\lo I$ and $\varphi_2:[1/2+\delta/2,1/2+\delta]\lo [1/2,1/2+\delta]$ are linear homeomorphisms with $\varphi_1(1/2)=0$ and $\varphi_2(1/2+\delta/2)=1/2$. By gluing lemma,
$f$ is continuous and hence is a loop such that $[f]=[\al*\bt*\al^{-1}]$. We show that $f\in\bigcap_{i=1}^n\lk K_i,U_i\rg$.

For every $i\in B$, $f(K_i)\sub U_i$ since $K_i\sub I\setminus[1/2-\delta,1/2+\delta]$, $f|_{I\setminus[1/2,1/2+\delta]}=\al*\al^{-1}|_{ I\setminus[1/2,1/2+\delta]}$ and $\al*\al^{-1}(K_i)\sub U_i$.\\
If $A\neq\varnothing$, then for every $i\in A$, $f(K_i)\sub U_i$ since $f|_{I\setminus[1/2,1/2+\delta]}=\al*\al^{-1}|_{ I\setminus[1/2,1/2+\delta]}$, $\al*\al^{-1}(K_i)\sub U_i$ and $f([1/2,1/2+\delta])\sub V\sub U_i$.\\
Therefore $f\in\bigcap_{i=1}^n\lk K_i,U_i\rg\sub q^{-1}(W)$ which implies that $[\al*\bt*\al^{-1}]=[f]\in W$.

For a given $g\in\psg$ we show that $g\in U$. There are finitely many generators of $\psg$, $[\al_i*\bt_i*\al_i^{-1}]$ for $i=1,2,...,m$, such that $g=[\al_1*\bt_1*\al_1^{-1}][\al_2*\bt_2*\al_2^{-1}]...[\al_m*\bt_m*\al_m^{-1}]$. Since $ U$ is an open neighborhood of $[e_x]$, by the above discussion $[\al_1*\bt_1*\al_1^{-1}]\in U$. $\pt$ is a homogeneous space (see \cite{Cal}) and hence $[\al_1*\bt_1*\al_1^{-1}]^{-1} U$ is an open neighborhood of $[e_x]$ which implies that $[\al_2*\bt_2*\al_2^{-1}]\in[\al_1*\bt_1*\al_1^{-1}]^{-1} U$. By a similar argument we have
 $$[\al_m*\bt_m*\al_m^{-1}]\in[\al_{m-1}*\bt_{m-1}*\al_{m-1}^{-1}]^{-1}...[\al_2*\bt_2*\al_2^{-1}]^{-1}[\al_1*\bt_1*\al_1^{-1}]^{-1} U$$
 and therefore $g\in U$. Moreover, since $\pt$ is a quasitopological group and hence all translations are homeomorphism, the previous result implies that $\pi_1^{sg}(X,x)\sub \ov{\{[e_{x}]\}}$.
\end{proof}
\begin{remark}
By \cite[Proposition 1.4.13]{A} the closure $\ov{H}$ of a subgroup $H$ of a quasitopological group $G$ is also a subgroup of $G$. As a consequence, it is routine to show that the closure $\ov{\{1\}}$ of the identity element of $G$ is always a closed normal subgroup of $G$, which equals $G$ if and only if $G$ has indiscrete topology. Also, it is easy to see that every nonempty closed set (and hence every nonempty open set) of $G$ is a disjoint union of cosets of the closure of the identity.
\end{remark}

The following corollaries are direct consequences of the inclusion $\pi_1^{sg}(X,x)\sub \ov{\{[e_{x}]\}}$ and the above remark.
\begin{corollary}
Every nonempty open or closed subset of $\pt$ is a disjoint union of some cosets of $\psg$ and also some cosets of $\ov{\psg}$..
\end{corollary}
\begin{corollary}
For a pointed topological space $(X,x)$, if $\{[e_x]\}$ is closed (or equivalently the topology of $\pt$ is $T_0$), then $X$ is homotopically Hausdorff.
\end{corollary}

\begin{remark}
Note that by Corollary 2.4 $\ov{\psg}$ is a union of some cosets of $\psg$ and $\ov{\psg}=\ov{\{[e_x]\}}$. Also, $\psg$ is a dense subgroup of $\pt$ if and only if $\pt$ is an indiscrete topological group.
\end{remark}

Let $G$ be a quasitopological group. If $\ov{\{1\}}$ has finite index in $G$, then $\ov{\{1\}}$ is both closed and open in $G$. Using this fact, $\ov{\psg}=\ov{\{[e_x]\}}$ and (1.1), we have the following corollaries.
\begin{corollary}
If $\ov{\psg}$ is a finite index subgroup of $\pt$, then $\pt$ is an indiscrete topological group if and only if $X$ has no non-trivial covering.
\end{corollary}

\begin{corollary}
If $\ov{\psg}$ is a finite index subgroup of $\pt$ and $\pt$ is connected, then $\pt$ is an indiscrete topological group and therefore $X$ has no non-trivial cover.
\end{corollary}
Note that if $\psg$ is a finite index subgroup of $\pt$ or $\pt$ is finite, then the above corollaries hold.

In the following example, we show that there exists a metric space $X$ such that $\psg$ is not closed and hence $\psg\neq\ov{\{[e_x]\}}$. Note that here $X$ is not locally path connected.
\begin{example}
Let $Y_1=\{(x,y)\in\mathbb{R}^2\ |\ x^2+y^2=(1/2+1/n)^2,\ n\in\N\}$, $Y_2=\{(x,y)\in\mathbb{R}^2\ |\ x^2+y^2=1/4\}\cup\{(x,0)\in\mathbb{R}^2\ |\ 1/2\leq x\leq 3/2\}$ and $Y=Y_1\cup Y_2$ with $x=(3/2,0)=a$ as the base point (see Figure 1). Let $f_i:S^1\lo S_i$ be a homeomorphism from the 1-sphere into $Y$ such that $f_i((1,0))=(1/2+1/i,0)$, where $S_i=\{(x,y)\in\mathbb{R}^2\ |\ x^2+y^2=(1/2+1/i)^2\}$, for every $i\in\N$.
 Put $X_0=Y$ and let $X_i=X_{i-1}\cup_{f_i}C_i$, where $C_i$ is a cone over $S^1$ with height $1$, be the space obtained by attaching the cone $C_i$ to $X_{i-1}$ via $f_i$, for all $i\in\N$. Consider $X=\bigcup_{i\in \N} X_i$, then $X$ is a connected, locally simply connected metric space and hence $\psg$ is trivial. Let $\al_i:I\lo [1/2+1/i,3/2]\times \{0\}\times \{0\}$ be a linear homeomorphism such that $\al_i(0)=a$ for every $i\in\N$, $\al:I\lo [1/2,3/2]\times \{0\}\times \{0\}$ be a linear homeomorphism such that $\al(0)=a$ and $f:I\lo X$ defined by $f(t)=1/2(cost,sint,0)$.
 Therefore the sequence $\al_n*f_n*\al_n^{-1}$ uniformly converges to $\al*f*\al^{-1}$. For every $n\in\N$, $\al_n*f_n*\al_n^{-1}$ is a null homotopic loop which implies that $1\neq[\al*f*\al^{-1}]\in\ov{\{[e_a]\}}=\ov{\psg}$. Note that the uniform topology and compact open topology are equivalent in the loop space of a metric space.
\begin{figure}
\begin{center}
 \includegraphics[scale=1]{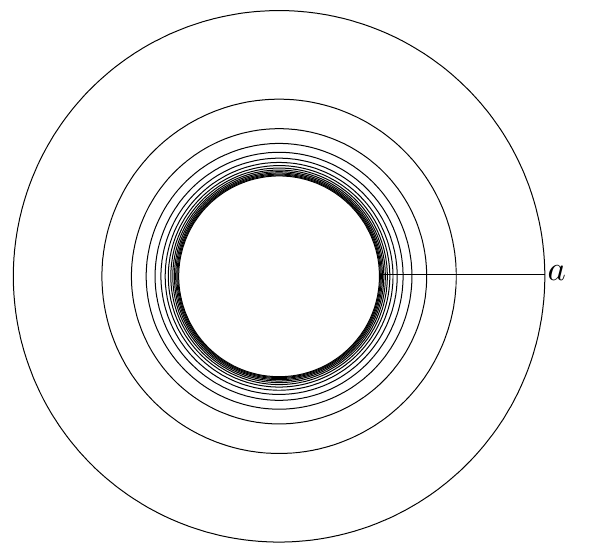}
  \caption{}\label{1}
 \end{center}
\end{figure}
\end{example}

\begin{definition}
We call a space $X$ small generated if $\pi_1(X,x)=\psg $, for every $x\in X$.
\end{definition}
\begin{remark}
By Theorem 2.2 every open neighborhood of trivial element of $\pt$ contains $\psg$. Hence if $X$ is a small generated space, then $\pt$ is an indiscrete topological group.
Note that using this fact, we have another proof for indiscreteness of $\pi_1^{qtop}(HA,0)$ since HA is a small generated space.
 It should be noted that the converse of the above fact is not necessarily true. For example, the space $X$ introduced in Example 2.9 is non-simply connected and homotopically Hausdorff space which implies that $\psg\neq\pi_1(X,x)$, but it has indiscrete topological fundamental group. Note that $[\al*f*\al^{-1}]$ generates $\pi_1(X,a)$ and by the argument at the end of Example 2.9 $[\al*f*\al^{-1}]\in\ov{\{[e_x]\}}$ which implies that $\pi_1(X,x)\sub \ov{\{[e_x]\}}$. Hence $\pt$ is an indiscrete topological group.
\end{remark}

A space $X$ is called a semi-locally small loop space if for each $x\in X$ there exists an open neighborhood $U$ of $x$ such that $i_*\pi_1(U,y)=\pi_1^s(X,y)$, for all $y\in U$, where $i:U\hookrightarrow X$ is the inclusion map (see \cite{P2}).
The authors \cite[Theorem 4.6]{P2} proved that for a connected, locally path connected and semi-locally small loop space $X$, $X$ is a small loop space if and only if $\pi_1^{qtop}(X,x)$ is an indiscrete topological group if and only if every covering $p:\wt{X}\lo X$ is trivial. Now, by the inclusion $\pi_1^{sg}(X,x)\sub \ov{\{[e_{x}]\}}$ it follows that density of $\pi_1^{sg}(X,x)$ in $\pi_1^{qtop}(X,x)$ is another equivalent condition to the above statements
\begin{remark}
Note that since the topology of $\pi_1^{\tau}(X,x)$ is coarser than the one of $\pt$, in fact $\pi_1^{\tau}(X,x)$ and $\pt$ have the same open subgroups \cite[Corollary 3.9]{Br2},
and $\pi_1^{\tau}(X,x)$ is always a topological group, it is routine to check that all the results of this section hold if we replace $\pt$ with $\pi_1^{\tau}(X,x)$. Also, note that the topological closure of $\psg$ in $\pt$ is a subset of the topological closure of $\psg$ in $\pi_1^{\tau}(X,x)$.
\end{remark}
\section{Small Generated Coverings}
By convention, the term \emph{universal covering} will always mean a categorical universal object, that is, a covering
$p:\wt{X}\lo X$, where $\wt{X}$ is path connected with the property that for every covering $q:\wt{Y}\lo X$ with a path connected space $\wt{Y}$ there exists a unique covering $r:\wt{X}\lo\wt{Y}$ such that $q\circ r= p$.
The following proposition was proved in \cite{P2}. It should be noted that the second inclusion also follows from \cite[\S2.5 Lemma 11]{S}.
\begin{proposition}
For every covering $p:\wt{X}\lo X$ and $x\in X$ the following relations hold: $$\pi_1^s(X,x)\leq\pi_1^{sg}(X,x)\leq p_*\pi_1(\wt{X},\tilde{x}).$$
\end{proposition}
Since the image subgroup $\pst$ in $\pi_1(X,x)$
consists of the homotopy classes of loops in $X$ based at $x$ whose lifts to $\wt{X}$ starting
at $\ti{x}$ are loops, for a covering $p:\wt{X}\lo X$ and $[\al]\in\psg$ every lift of $\al$ in $\wt{X}$ is a loop.
The following proposition comes from the local homeomorphism
property of $p$ and the injectivity of $p_*$. (cf. proof of \cite[Theorem 4.7]{P2}).
\begin{proposition}
If $\px$ is a covering, then $p_*\pi_1^{sg}(\wt{X},\ti{x})=\psg$.
\end{proposition}

By the above result every covering space of a small generated space $X$ is homeomorphic to $X$.
\begin{definition}
By a $\tx{small generated covering}$ of a topological space $X$ we mean a covering $p:\wt{X}\lo X$ such that $\wt{X}$ is a small generated space.
\end{definition}

The following corollary is an immediate consequence of Proposition 3.2 and Definition 3.3.
\begin{corollary}
A covering $p:\wt{X}\lo X$ is a small generated covering if and only if $\pi_1^{sg}(X,x)=p_*\pi_1(\wt{X},\ti{x})$.
\end{corollary}

In classical covering theory, for a connected and locally path connected space $X$, the existence of simply connected (universal) covering $\px$ is equivalent to the semi-locally simply connectedness of $X$ that means for every $y\in X$ there exists a neighborhood $U$ of $y$ such that $i_*\pi_1(U,y)\leq \pst=1$. By \cite[\S2.5 Theorem 13]{S} , for a connected and locally path connected $X$, there is a covering $\px$ with $\pst=\psg$
if there is an open cover $\cal U$ of $X$ such that $\pi({\cal U},x)\subseteq \psg$. Since $\pi_1^{sg}(X,\alpha(0))=[\alpha]\pi_1^{sg}(X,\alpha(1))[\alpha^{-1}]$ for all paths $\alpha$ in $X$, we may combine this fact with (*) to obtain the following criterion:\\
For a connected and locally path connected topological space $X$, there is a small generated covering $\px$ if and only if every $y\in X$ has an open neighborhood $U$ in $X$ such that $i_*\pi_1(U,y)\subseteq\pi_1^{sg}(X,y)$, where $i:U\hookrightarrow X$ denotes inclusion. (Note that for path connected $U$, $i_*\pi_1(U,u)\subseteq\pi_1^{sg}(X,u)$ holds for some $u\in U$ if and only if it holds for all $u\in U$.) Moreover, by (*), a small generated covering of $X$ satisfies the lifting criterion \cite[\S2.4 Theorem 5 and \S2.5 Lemma 1]{S} to all other coverings of $X$ and hence is a universal covering of $X$.
\begin{definition}
 We call a space $X$ semi-locally small generated if and only if for each $x\in X$ there exists an open neighborhood $U$ of $x$ such that $i_*\pi_1(U,x)\leq\pi_1^{sg}(X,x)$, where $i:U\hookrightarrow X$ is the inclusion map.
\end{definition}
\begin{theorem}
A connected, locally path connected space $X$ has a small generated covering if and only if $X$ is a semi-locally small generated space.
Also, a small generated covering of $X$ is a universal covering of $X$.
\end{theorem}
\begin{example}
Every small generated space is semi-locally small generated. Also, the product $X\times Y$ is semi-locally small generated if both $X$ and $Y$ are semi-locally small generated.
 If $(X,x)$ is a pointed small generated space and $(Y,y)$ is first countable and locally simply connected at $y$, then the one point union $X\vee Y=\frac{X\cup Y}{x\sim y}$ is semi-locally small generated.
\end{example}
Since every null homotopic loop is also a small loop, semi-locally simply connected spaces are also semi-locally small generated. Calcut and McCarthy \cite{Cal} proved that for connected and locally path connected spaces, semi-locally simply connectedness is equivalent to openness of $[e_x]$ that makes $\pt$ a discrete topological group. Here we extend this result in the following theorem using Corollary 2.4 and (1.1).
\begin{theorem}
Let $X$ be a connected and locally path connected space. Then the following statements are equivalent:\\
(i) $X$ is semi-locally small generated.\\
(ii) For each $x\in X$, every nonempty subset $U$ of $\pt$ is open if and only if $U$ is a union of some cosets of $\psg$.
\end{theorem}

In the following, we give an example of a small generated covering.
\begin{example}
Let $(S^{1},0)$ be the unique circle, $(HA,x)$ be the Harmonic Archipelago, where $x$ is the common point of boundary circles and $X=S^1\vee HA $ be their one point union. Consider the space $Y={\mathbb{R}\bigcup(\mathbb{Z}\times HA)}$ with the equivalence relation ${n\sim (n,x)}$, for every $n\in\mathbb{Z}$ and let $\wt{X}=Y/\sim$. Define $p:\wt{X}\lo X$ such that $p$ wraps $\mathbb{R}$ around $S^1$ like exponential map and sends identically Harmonic Archipelago to itself at every integer. Obviously $p:\wt{X}\lo X$ is a covering and since every loop in $\wt{X}$ is small generated, $\wt{X}$ is a small generated space and hence $p$ is a small generated covering.
\end{example}

The following corollary follows from the inclusion $\psg\sub\ov{\{[e_x]\}}$, Theorem 3.6 and (1.1).
\begin{corollary}
If $X$ is a connected, locally path connected and semi-locally small generated space, then the following statements are equivalent.\\
(i) $\psg$ is a dense subgroup of $\pt$.\\
(ii) $\pi_1^{top}(X,x)$ is an indiscrete topological group.\\
(iii) every covering $p:\wt{X}\lo X$ is trivial.\\
(iv) $X$ is a small generated space.
\end{corollary}
\section{An Application to Quasitopological Fundamental Groups}
After that Fabel \cite{F} and Brazas \cite{Br} showed that the quasitopological fundamental group introduced by Biss \cite{B} fails to be a topological group, in general,
there is an open question that when quasitopological fundamental groups are topological groups. Calcut and McCarthy \cite{Cal} proved that the topology of fundamental group of a locally path connected and semi-locally simply connected space is discrete and so this space has the quasitopological fundamental group as topological group.
The counterexamples of Fabel \cite{F} and Brazas \cite{Br} show that $\pi_1^{qtop}$ is not a functor into the category of topological groups. Brazas \cite{Br2} introduced a new topology on fundamental groups made them topological groups and denoted this new functor by $\pi_1^{\tau}$. For a topological space $X$, $\pt$ and $\pi_1^{\tau}(X,x)$ has the same underlying set and algebraic structure but different topologies. In fact, the topology of $\pi_1^{\tau}(X,x)$ is obtained by removing some open subsets of $\pt$ and hence the topology of $\pi_1^{\tau}(X,x)$ is coarser than the topology of $\pt$. Since it is not known which open subsets of $\pt$ are removed, working with $\pt$ seems easier. Also, if $\pt$ is a topological group, then $\pt\cong\pi_1^{\tau}(X,x)$ as topological groups \cite{Br}. Therefore, the question\emph{``What kind of topological structure is $\pt$?''} is still interesting.
In the following theorem we show that quasitopological fundamental groups of semi-locally small generated spaces are topological groups.
\begin{theorem}
The quasitopological fundamental group of a semi-locally small generated space is a topological group.
\end{theorem}
\begin{proof}
It is sufficient to show that the group multiplication is continuous. Assume that $X$ is a semi-locally small generated space and $\mu:\pt\times\pt\lo\pt$ is the group multiplication. Let $U$ be an open neighborhood of $[\al][\bt]=[\alpha *\beta]$ for $[\al], [\bt]\in\pt$, where $*$ is the operation of concatenation of two paths. We show that there are open neighborhoods $V$ and $W$ of $\al$ and $\bt$, respectively, such that $\mu(V,W)\sub U$. Since every open subset of $\pt$ is a union of some cosets of $\psg$, $([\al][\bt])\psg\sub U$. By Theorem 3.8, $\psg$ is open and hence $V=[\al]\psg$ and $W=[\bt]\psg$ are open neighborhoods of $[\al]$ and $[\bt]$, respectively. Also, $\mu([\al]\psg,[\bt]\psg)=([\al][\bt])\psg$ by normality of $\psg$, as desired.
\end{proof}
\begin{remark}
Note that if $X$ is semi-locally small generated, then by the above theorem and the definition of $\pi_1^{\tau}(X,x)$ we have $\pt=\pi_1^{\tau}(X,x)$.
Hence Theorem 3.8 and Corollary 3.10 hold if we replace $\pt$ with $\pi_1^{\tau}(X,x)$.
\end{remark}

Since every semi-locally small loop space is also semi-locally small generated, then we have the following result.
\begin{corollary}
The quasitopological fundamental group of a semi-locally small loop space is a topological group.
\end{corollary}

The authors \cite{P1} showed that the quasitopological fundamental groups of small loop spaces are indiscrete topological groups (note that this fact also follows from the inclusion $\psg\sub\ov{\{[e_x]\}}$). Also, therein, using the results of \cite{G}, the authors introduced a class of spaces with quasitopological fundamental groups as prodiscrete topological groups. Note that all known quasitopological fundamental groups which are topological group have discrete, indiscrete or prodiscrete topology. In the next example, we show that the quasitopological fundamental groups of semi-locally small generated spaces do not have necessarily these topologies.
\begin{example}
By Example 3.9, the space $X=S^1\vee HA$ has small generated covering and hence $\pt$ is a topological group. Since $X$ is a metric space and the uniform topology and the compact open topology are equivalent in metric spaces, $\ov{\psg}\neq \pt$ which implies that the topology of $\pt$ is not indiscrete. Also, by a corollary of \cite[III.7.3, Proposition 2]{Bu} prodiscrete topological groups are totally disconnected and hence $\pt$ is not a prodiscrete topological group since it is not totally disconnected by Theorem 2.2.
\end{example}
\subsection*{Acknowledgements}
The authors would like to thank the referee for the valuable comments and useful suggestions to improve the present paper.






\section*{References}
\begin{thebibliography}{99}

\bibitem{A} ARHANGEL'SKII, A.---TKACHENKO, M.: \textit{Topological Groups and Related Structures}, Atlantis Studies in Mathematics,
2008.

\bibitem{B} BISS, D.: \textit{The topological fundamental group and generalized covering spaces},  Topology and
its Applications \textbf{124} (2002), 355-371.

\bibitem{Bu} BOURBAKI, N.: \textit{Elements of Mathematics, General Topology I}, Addison-Wesley, London, 1966.

\bibitem{Br} BRAZAS, J.: \textit{The topological fundamental group and free topological groups},
Topology and its Applications \textbf{158} (2011), 779-802.

\bibitem{Br2} BRAZAS, J.: \textit{The fundamental group as topological group}, arXiv:1009.3972v5.

\bibitem{Cal} CALCUT, J.S.---MCCARTHY, J.D.:  \textit{Discreteness and homogeneity of the topological fundamental
group}, Topology Proc. \textbf{34} (2009), 339-349.

\bibitem{F} FABEL, P.: \textit{Multiplication is discontinuous in the Hawaiian earring group (with the quotient topology)},
Bull. Polish. Acad. Math. \textbf{59} (2011), 77-83.

\bibitem{FZ} FISCHER, H.---ZASTROW, A.: \textit{Generalized universal coverings and the shape group}, Fund. Math. \textbf{197} (2007), 167-196.

\bibitem{G} GHANE, H.---HAMED, Z.---MASHAYEKHY, B.---MIREBRAHIMI, H.: \textit{On topological homotopy groups of n-Hawaiian like spaces},
Topology Proc. \textbf{36} (2010), 255-266.

\bibitem{P1} PAKDAMAN, A.---TORABI, H.---MASHAYEKHY, B.: \textit{On H-groups and their applications to topological fundamental group}, arXiv:1009.5176v1.

\bibitem{P2} PAKDAMAN, A.---TORABI, H.---MASHAYEKHY, B.: \textit{Small loop spaces and covering theory of non-homotopically Hausdorff spaces}, Topology and its Application \textbf{158} (2011), 803-809.

\bibitem{T} TORABI, H.---PAKDAMAN, A.---MASHAYEKHY, B.: \textit{On the Spanier groups and covering and semicovering spaces}, arXiv:1207.4394v1.

\bibitem{S} SPANIER, E.H.: \textit{Algebraic Topology}, McGraw-Hill, New York, 1966.

\bibitem{V} VIRK, Z.: \textit{Small loop spaces}, Topology and its Applications \textbf{157} (2010), 451-455.


\end{thebibliography}







\end{document}